\documentclass[11pt]{amsart}

\usepackage{amsmath}
\usepackage{amssymb}
\usepackage{commath}
\usepackage{mathtools}

\usepackage{hyperref}
\usepackage{xcolor}

\theoremstyle{plain}
\newtheorem{theorem}{Theorem}
\newtheorem{lemma}[theorem]{Lemma}
\newtheorem{proposition}[theorem]{Proposition}
\newtheorem{corollary}[theorem]{Corollary}

\numberwithin{theorem}{section}
\numberwithin{equation}{section}

\newcommand{\C}{\mathbb{C}}
\newcommand{\N}{\mathbb{N}}
\newcommand{\rN}{\overrightarrow{\mathbb{N}}}
\newcommand{\R}{\mathbb{R}}
\newcommand{\T}{\mathbb{T}}
\newcommand{\Z}{\mathbb{Z}}

\newcommand{\GL}{\mathrm{GL}}

\newcommand{\SU}{\mathrm{SU}}
\newcommand{\U}{\mathrm{U}}
\newcommand{\Spe}{\mathrm{S}}

\newcommand{\fu}{\mathfrak{u}}
\newcommand{\fsl}{\mathfrak{sl}}

\newcommand{\cH}{\mathcal{H}}
\newcommand{\cT}{\mathcal{T}}
\newcommand{\cA}{\mathcal{A}}
\newcommand{\cP}{\mathcal{P}}

\DeclareMathOperator{\diag}{diag}

\begin{document}

\title[Toeplitz operators with $\U(2)\times\T^2$-invariant symbols]{Toeplitz operators on the domain $\{Z\in M_{2\times2}(\C) \mid Z Z^* < I\}$ with $\U(2)\times\T^2$-invariant symbols}

\author{Matthew Dawson}
\address{CONACYT--CIMAT Unidad M\'erida, Parque Cient\'ifico y Tecnol\'ogico de Yucat\'an, Km 5.5 Carretera Sierra Papacal -- Chuburn\'a Puerto, Sierra Papacal, Mérida, Yucat\'an 97302, M\'exico}
\email{matthew.dawson@cimat.mx}

\author{Gestur \'{O}lafsson}
\address{Department of Mathematics, Louisiana State University, Baton Rouge, LA 70803, U.S.A.}
\email{olafsson@math.lsu.edu}

\author{Raul Quiroga-Barranco}
\address{Centro de Investigaci\'{o}n en Matem\'{a}ticas,  Jalisco s/n, Col. Valenciana, Guanajuato, GTO 36240, M\'{e}xico}
\email{quiroga@cimat.mx}

\thanks{The research of G. \'Olafsson was partially supported by Simon grant 586106. The research of Raul Quiroga-Barranco was partially supported by a Conacyt grant.}

\begin{abstract}
Let $D$ be the irreducible bounded symmetric domain of $2\times2$ complex matrices that satisfy $ZZ^* < I_2$. The biholomorphism group of $D$ is realized by $\U(2,2)$ with isotropy at the origin given by $\U(2)\times\U(2)$. Denote by $\T^2$ the subgroup of diagonal matrices in $\U(2)$. We prove that the set of $\U(2)\times\T^2$-invariant essentially bounded symbols yield Toeplitz operators that generate commutative $C^*$-algebras on all weighted Bergman spaces over $D$.  Using tools from representation theory, we also provide an integral formula for the spectra of these Toeplitz operators.
\end{abstract}

\maketitle

\textit{Dedicated to Nikolai Vasilevski on the occasion of his 70th birthday}

\section{Introduction}
\noindent
In this work we consider the problem of the existence of commutative $C^*$-algebras that are generated by families of Toeplitz operators on weighted Bergman spaces over irreducible bounded symmetric domains. More precisely, we are interested in the case where the Toeplitz operators are those given by symbols invariant by some closed subgroup of the group of biholomorphisms. This problem has turned out to be a quite interesting one thanks in part to the application of representation theory.

An important particular case is given when one considers the subgroup fixing some point in the domain, in other words, a maximal compact subgroup of the group of biholomorphisms. In \cite{DOQ}, we proved that for such maximal compact subgroups, the corresponding $C^*$-algebra is commutative. On the other hand, there is another interesting family of subgroups to consider: the maximal tori in the group of biholomorphisms. By the results from \cite{DOQ} it is straightforward to check that the $C^*$-algebra generated by the Toeplitz operators whose symbols are invariant under a fixed maximal torus is commutative if and only if the irreducible bounded symmetric domain is biholomorphically equivalent to some unit ball.

These results have inspired Nikolai Vasilevski to pose the following question. Let $D$ be an irreducible bounded symmetric domain that is not biholomorphically equivalent to a unit ball (that is, it is not of rank one), $K$ a maximal compact subgroup and $T$ a maximal torus in the group of biholomorphisms of $D$. Does there exist a closed subgroup $H$ such that $T \subsetneq H \subsetneq K$ for which the $C^*$-algebras (for all weights) generated by Toeplitz operators with $H$-invariant symbols are commutative? The goal of this work is to give a positive answer to this question for the classical Cartan domain of type $I$ of $2\times 2$ matrices. In the rest of this work we will denote simply by $D$ this domain.

The group of biholomorphisms of $D$ is realized by the Lie group $\U(2,2)$ acting by fractional linear transformations. A maximal compact subgroup is given by $\U(2)\times\U(2)$, which contains the maximal torus $\T^2\times\T^2$, where $\T^2$ denotes the group of $2\times2$ diagonal matrices with diagonal entries in $\T$. We prove that there are exactly two subgroups properly between $\U(2)\times\U(2)$ and $\T^2\times\T^2$, and these are $\U(2)\times\T^2$ and $\T^2\times\U(2)$ (see Proposition~\ref{prop:subgroupsT4U2U2}), for which it is also proved that the corresponding $C^*$-algebras generated by Toeplitz operators are unitarily equivalent (see Proposition~\ref{prop:U2T2vsT2U2}). In Section~\ref{sec:U(2)T2} we study the properties of $\U(2)\times\T^2$-invariant symbols. The main result here is Theorem~\ref{thm:Toeplitz-U2T2}, where we prove the commutativity of the $C^*$-algebras generated by Toeplitz operators whose symbols are $\U(2)\times\T^2$-invariant. As a first step to understand the structure of these $C^*$-algebras we provide in Section~\ref{sec:spectra} a computation of the spectra of the Toeplitz operators. The main result here is Theorem~\ref{thm:coefficients}.

We would like to use this opportunity to thank Nikolai Vasilevski, to whom this work is dedicated. Nikolai has been a very good friend and an excellent collaborator. He has provided us all with many ideas to work with.

\section{Preliminaries}\label{sec:preliminaries}
\noindent
Let us consider the classical Cartan domain given by
\[D  = \{ Z \in M_{2\times2}(\C) : Z Z^* < I_2 \},
\]
where $A<B$ means that $B-A $ is positive definite. This domain is sometimes denoted by either $D^I_{2,2}$ or $D_{2,2}$.

We consider the Lie groups
\[
    \U(2,2) = \{ M \in \GL(4,\C) : M^* I_{2,2} M = I_{2,2} \},\]
where
\[
    I_{2,2} =
    \begin{pmatrix}
      I_2 & 0 \\
      0 & -I_2
    \end{pmatrix}.
\]
and the Lie group
\[ \SU (2,2)  =\{M\in\U (2,2) : \det M =1\}.\]
Then $\SU (2,2)$, and hence also $\U(2,2)$, act transitively on $D$ by
\[
    \begin{pmatrix}
      A & B \\
      C & D
    \end{pmatrix}\cdot Z = (AZ+B)(CZ+D)^{-1},
\]
where we have a block decomposition by matrices with size $2\times2$. And $\SU (2,2)$ is, up to covering, the
group of biholomorphic isometries of $D$ and the action of $\SU (2,2)$ is locally faithful.
We observe that the action of $\U(2,2)$ on $D$ is not faithful. More precisely, the kernel of its action is the subgroup of matrices of the form $tI_4$, where $t \in \T$.

The maximal compact subgroup of $\U(2,2)$ that fixes the origin $0$ in $D$ is given by
\[
    \U(2)\times\U(2) = \left\{
        \begin{pmatrix}
          A & 0 \\
          0 & B
        \end{pmatrix} : A\in\U(2), B\in\U(2)
    \right\}.
\]
For simplicity, we write the elements of $\U(2)\times\U(2)$ as $(A,B)$ instead of using their
block diagonal representation. A maximal torus of $\U(2)\times\U(2)$ is given by
\[
    \T^4 = \{ (D_1, D_2) \in \U(2)\times\U(2) :
            D_1, D_2 \text{ diagonal} \}.
\]
The corresponding maximal compact subgroup and maximal torus in $\SU (2,2)$ are given by
\begin{align*}
   \Spe(\U(2)\times\U(2)) &= \{(A,B) \in \U(2)\times\U(2) :
        \det(A)\det(B) =1 \}, \\
    \T^3 &= \{(D_1,D_2) \in \T^4 : \det(D_1)\det(D_2) = 1\}.
\end{align*}

For every $\lambda > 3$ we will consider the weighted measure $v_\lambda$ on $D$ given by
\[
    \dif v_\lambda(Z) = c_\lambda \det(I_2 - Z Z^*)^{\lambda - 4}\dif Z
\]
where the constant $c_\lambda$ is chosen so that $v_\lambda$ is a probability measure. In particular, we have,
see \cite[Thm. 2.2.1]{Hua63}:
\[
    c_\lambda =
    \frac{ (\lambda -3)(\lambda-2)^2(\lambda-1)}{\pi^4}, \quad \lambda >3.\]

The Hilbert space inner product defined by $v_\lambda$ will be denoted by $\left<\cdot,\cdot\right>_\lambda$. We will from now on always assume that $\lambda >3$. The weighted Bergman space $\cH^2_\lambda(D)$ is the Hilbert space of holomorphic functions that belong to $L^2(D,v_\lambda)$. This is a reproducing kernel Hilbert space with Bergman kernel given by
\[
    k_\lambda(Z,W) = \det(I_2 - Z W^*)^{-\lambda},
\]
which yields the Bergman projection $B_\lambda : L^2(D,v_\lambda) \rightarrow \cH^2_\lambda(D)$ given by
\[B_\lambda f(Z)=\int_D f(W)k_\lambda (Z,W)dv_\lambda (W).\]

We recall that the space of holomorphic polynomials $\cP(M_{2\times2}(\C))$ is dense on every weighted Bergman space. Furthermore, it is well known that one has, for every $\lambda > 3$, the decomposition
\[
    \cH^2_\lambda(D) = \bigoplus_{d=0}^\infty \cP^d(M_{2\times2}(\C))
\]
into a direct sum of Hilbert spaces, where $\cP^d(M_{2\times2}(\C))$ denotes the subspace of homogeneous holomorphic polynomials of degree $d$.

For every essentially bounded symbol $\varphi \in L^\infty(D)$ and for every $\lambda > 3$ we define the corresponding Toeplitz operator by
\[
    T^{(\lambda)}_\varphi(f) = B_\lambda(\varphi f), \quad f\in\cH_\lambda^2(D) .\]

 In particular, these Toeplitz operators are given by the following expression
\[
    T^{(\lambda)}_\varphi(f)(Z) =
        c_\lambda\int_{D}
        \frac{\varphi(W)f(W)\det(I_2-WW^*)^{\lambda-4}}{\det(I_2-ZW^*)^\lambda}\, \dif W.
\]

On the other hand, for every $\lambda > 3$ there is an irreducible unitary representation of $\U (2,2)$ acting
on $\cH_\lambda^2(D)$ given by
\begin{align*}
  \pi_\lambda : \widetilde{\U}(2,2) \times \cH^2_\lambda(D) &\rightarrow
        \cH^2_\lambda(D) \\
  (\pi_\lambda(g)f)(Z) &= j(g^{-1},Z)^\frac{\lambda}{4} f(g^{-1}Z),
\end{align*}
where $j(g,Z)$ denotes the complex Jacobian of the transformation $g$ at the point $Z$.

We note that every $g \in \U(2)\times\U(2)$ defines a linear unitary transformation of $D$ that preserves all the measures $\dif v_\lambda$.

If $\lambda/4$ is not an integer, then $j(g,Z)^{\lambda/4}$ is not always well defined which makes it necessary to consider a covering of $\U(2,2)$. We therefore consider the universal covering group $\widetilde\U (2,2)$ of $\U (2,2)$ and its subgroup    $\R\times\SU(2)\times\R\times\SU(2)$, the universal covering group of $\U(2)\times\U(2)$. Here the covering map is given by
\[
    (x,A,y,B) \mapsto (e^{ix}A,e^{iy}B).
\]
Hence, the action of $\R\times\SU(2)\times\R\times\SU(2)$ on $D$ is given by the expression
\[
    (x,A,y,B)Z = e^{i(x-y)}AZB^{-1}.
\]
It follows that the restriction of $\pi_\lambda$ to the subgroup $\R\times\SU(2)\times\R\times\SU(2)$ is given by the expression
\[
    (\pi_\lambda(x,A,y,B)f)(Z) = e^{i\lambda(y-x)}f(e^{i(y-x)}A^{-1}ZB).
\]
It is well known that this restriction is multiplicity-free for every $\lambda>3$ (see \cite{DOQ} and \cite{K08}).

It is useful to consider as well the representation
\begin{align*}
  \pi_\lambda' : (U(2)\times\U(2)) \times \cH^2_\lambda(D) &\rightarrow
        \cH^2_\lambda(D) \\
  (\pi_\lambda'(g)f)(Z) &= f(g^{-1}Z),
\end{align*}
which is well-defined and unitary as a consequence of the previous remarks.
Note that the representations $\pi_\lambda$ and $\pi_\lambda'$ are defined on groups that differ by a covering, but they also differ by the factor $e^{i\lambda(y-x)}$. It follows that $\pi_\lambda'$ is multiplicity-free with the same isotypic decomposition as that of $\pi_\lambda$.

\section{Toeplitz operators invariant under subgroups of $\U(2)\times\U(2)$}
\noindent
For a closed subgroup $H \subset \U(2)\times\U(2)$ we will denote by $\cA^H$ the complex vector space of essentially bounded symbols $\varphi$ on $D$ that are $H$-invariant, i.e.~such that for every $h \in H$ we have
\[
    \varphi(hZ) = \varphi(Z)
\]
for almost every $Z \in D$. Denote by $\cT^{(\lambda)}(\cA^H)$ the $C^*$-algebra generated by Toeplitz operators with symbols in $\cA^H$ acting on the weighted Bergman space $\cH^2_\lambda(D)$. We have $\U (2)\times \U(2)= \T (\Spe(\U (2)\times \U(2)))$ and the center acts trivially on $D$. We also point to the special case that will be the main topic of this article.

Let us denote
\[
    \U(2) \times \T = \left\{
       (A, t) = \left(A,
        \begin{pmatrix}
          t & 0 \\
          0 & \overline{t}
        \end{pmatrix}
        \right) : A \in \U(2), t \in \T
        \right\}.
\]
We now prove that $\U(2)\times\T$-invariance is equivalent to $\U(2)\times\T^2$-invariance.

\begin{lemma}\label{lem:U2T-invariance}
    The groups $\U(2)\times\T^2$ and $\U(2)\times\T$ have the same orbits. In other words, for every $Z \in D$, we have
    \[
        (\U(2)\times\T) Z = (\U(2)\times\T^2) Z.
    \]
    In particular, an essentially bounded symbol $\varphi$ is $\U(2)\times\T^2$-invariant if and only if it is $\U(2)\times\T$-invariant.
\end{lemma}
\begin{proof}
    We observe that $\U(2)\times\T^2$ is generated as a group by $\U(2)\times\T$ and the subgroup
    \[
        \{I_2\}\times\T I_2.
    \]
    But for every $t \in \T$ and $Z \in D$ we have
    \[
        (I_2,tI_2)Z = \overline{t}Z = (\overline{t}I_2,I_2)Z
    \]
    which is a biholomorphism of $D$ already realized by elements of $\U(2)\times\T$. Hence, $\U(2)\times\T^2$ and $\U(2)\times\T$ yield the same transformations on their actions on $D$, and so the result follows.
\end{proof}

The following is now a particular case of \cite[Thm. 6.4]{DOQ}
and can be proved directly in exactly the same way.

\begin{theorem}\label{thm:H-commutativeC*}
    For a closed subgroup $H$ of $\U(2)\times\U(2)$ the following conditions are equivalent for every $\lambda > 3$:
    \begin{enumerate}
      \item The $C^*$-algebra $\cT^{(\lambda)}(\cA^H)$ is commutative.
      \item The restriction $\pi_\lambda|_H$ is multiplicity-free.
    \end{enumerate}
\end{theorem}

As noted in Section~\ref{sec:preliminaries}, the unitary representation $\pi_\lambda$ is multiplicity-free on $\Spe(\U(2)\times\U(2))$ and thus the $C^*$-algebra generated by Toeplitz operators by $\Spe(\U(2)\times\U(2))$-invariant symbols is commutative for every weight $\lambda > 3$. Such operators are also known as radial Toeplitz operators.

On the other hand, it follows from Example~6.5 from \cite{DOQ} that the restriction $\pi_\lambda|_{\T^3}$ is not multiplicity-free, where $\T^3$ is the maximal torus of $\Spe(\U(2)\times\U(2))$ described in Section~\ref{sec:preliminaries}. Hence, we conclude that $\cT^{(\lambda)}(\cA^{\T^3})$ is not commutative for any $\lambda > 3$.

We now consider subgroups $H$ such that $\T^3 \subset H \subset \Spe(\U(2)\times\U(2))$ or, equivalently, subgroups $H$ such that $\T^4 \subset H \subset \U(2)\times\U(2)$. For simplicity, we will assume that $H$ is connected.

\begin{proposition}\label{prop:subgroupsT4U2U2}
    Let $\T^4$ denote the subgroup of diagonal matrices in $\U(2)\times\U(2)$. Then the only connected subgroups strictly between $\U(2)\times\U(2)$ and $\T^4$ are $\U(2)\times\T^2$ and $\T^2\times\U(2)$. In particular, the only connected subgroups strictly between $\Spe(\U(2)\times\U(2))$ and $\T^3$ are $\Spe(\U(2)\times\T^2)$ and $\Spe(\T^2\times\U(2))$.
\end{proposition}
\begin{proof}
    It is enough to prove the first claim for the corresponding Lie algebras.

    First note that ($x_1, x_2 \in \R, z \in \C$)
    \[
        \left[
        \begin{pmatrix}
          ix_1 & 0 \\
          0 & ix_2
        \end{pmatrix},
        \begin{pmatrix}
          0 & z \\
          -\overline{z} & 0
        \end{pmatrix}
        \right]
        =
        \begin{pmatrix}
          0 & i(x_1-x_2)z \\
          -\overline{i(x_1-x_2)z} & 0
        \end{pmatrix},
    \]
    which proves that the space
    \[
        V =
        \left\{
            \begin{pmatrix}
              0 & z \\
              -\overline{z} & 0
            \end{pmatrix} : z \in \C
        \right\}
    \]
    is an irreducible $i\R^2$-submodule of $\fu(2)$. Hence, the decomposition of $\fu(2)\times\fu(2)$ into irreducible $i\R^4$-submodules is given by
    \[
        \fu(2)\times\fu(2) = i\R^4 \oplus V\times\{0\}
                    \oplus \{0\} \times V.
    \]

    We conclude that $\fu(2)\times i\R^2$ and $i\R^2\times\fu(2)$ are the only $i\R^4$-submodules strictly between $\fu(2)\times\fu(2)$ and $i\R^4$, and both are Lie algebras.
\end{proof}

There is natural biholomorphism
\begin{align*}
    F : D &\rightarrow D  \\
    Z &\mapsto Z^\top
\end{align*}
that clearly preserves all the weighted measures $\dif v_\lambda$. Hence, $F$ induces a unitary map
\begin{align*}
    F^* : L^2(D,v_\lambda) &\rightarrow L^2(D,v_\lambda) \\
    F^*(f) &= f\circ F^{-1}
\end{align*}
that preserves $\cH^2_\lambda(D)$. And the same expression
\[
    \varphi \mapsto F^*(\varphi) = \varphi\circ F^{-1}
\]
defines an isometric isomorphism on the space $L^\infty(D)$ of essentially bounded symbols.

Furthermore, we consider the automorphism $\rho \in \mathrm{Aut}(\U(2)\times\U(2))$ given by
$\rho(A,B) = (\overline{B},\overline{A})$. Thus, we clearly have
\[
    F((A,B)Z) = F(AZB^{-1}) = \overline{B}Z^\top \overline{A}^{-1}
    = \rho(A,B) F(Z),
\]
for all $(A,B) \in \U(2)\times\U(2)$ and $Z \in D$. In other words, the map $F$ intertwines the $\U(2)\times\U(2)$-action with that of the image of $\rho$.

We observe that $\rho(\U(2)\times\T^2) = \T^2\times\U(2)$. Hence, the previous constructions can be used to prove that both groups define equivalent $C^*$-algebras from invariant Toeplitz operators.

\begin{proposition}\label{prop:U2T2vsT2U2}
    The isomorphism of $L^\infty(D)$ given by $F^*$ maps $\cA^{\U(2)\times\T^2}$ onto $\cA^{\T^2\times\U(2)}$. Furthermore, for every weight $\lambda > 3$ and for every $\varphi \in \cA^{\U(2)\times\T^2}$ we have
    \[
        T^{(\lambda)}_{F^*(\varphi)} = F^*\circ T^{(\lambda)}_\varphi
            \circ (F^*)^{-1}.
    \]
    In particular, the $C^*$-algebras $\cT^{(\lambda)}(\cA^{\U(2)\times\T^2})$ and $\cT^{(\lambda)}(\cA^{\T^2\times\U(2)})$ are unitarily equivalent for every $\lambda > 3$.
\end{proposition}
\begin{proof}
    From the above computations, for a given $\varphi \in L^\infty(D)$ we have
    \[
        \varphi \circ (A,B) \circ F^{-1}
            = \varphi \circ F^{-1} \circ \rho(A,B)
    \]
    for every $(A,B) \in \U(2)\times\U(2)$. Hence, $\varphi$ is $U(2)\times\T^2$-invariant if and only if $F^*(\varphi)$ is $\T^2\times\U(2)$-invariant. This proves the first part.

    On the other hand, we use that the map $F^*$ is unitary on $L^2(D,v_\lambda)$ to conclude that for every $f,g \in \cH^2_\lambda(D)$ we have
    \begin{align*}
        \left<T^{(\lambda)}_{F^*(\varphi)}(f),g\right>_\lambda
            &= \left<F^*(\varphi)f,g\right>_\lambda \\
            &= \left<(\varphi\circ F^{-1})f,g\right>_\lambda \\
            &= \left<\varphi (f\circ F),g\circ F\right>_\lambda \\
            &= \left<T^{(\lambda)}_{\varphi} \circ (F^*)^{-1}(f),
                (F^*)^{-1}g\right>_\lambda \\
            &= \left<F^* \circ T^{(\lambda)}_{\varphi} \circ (F^*)^{-1}(f),
                g\right>_\lambda,
    \end{align*}
    and this completes the proof.
\end{proof}

\section{$\U(2)\times\T^2$-invariant symbols}\label{sec:U(2)T2}
\noindent
As noted in Section~\ref{sec:preliminaries}, the subgroup $\U(2)\times\U(2)$ does not act faithfully. Hence, it is convenient to consider suitable subgroups for which the action is at least locally faithful. This is particularly important when describing the orbits of the subgroups considered. We also noted before that the most natural choice is to consider subgroups of $\Spe(\U(2)\times\U(2))$, however for our setup it will be useful to consider other subgroups.

For the case of the subgroup $\U(2)\times\T^2$ it turns out that $\U(2)\times\T^2$-invariance is equivalent to $\Spe(\U(2)\times\T^2)$-invariance. This holds for the action through biholomorphisms on $D$ and so for every induced action on function spaces over $D$.

To understand the structure of the $\U(2)\times\T$-orbits the next result provides a choice of a canonical element on each orbit.

\begin{proposition}\label{prop:U2T-orbits}
    For every $Z \in M_{2\times2}(\C)$ there exists $r \in [0,\infty)^3$ and $(A,t) \in \U(2)\times\T$ such that
    \[
        (A,t) Z =
        \begin{pmatrix}
          r_1 & r_2 \\
          0 & r_3
        \end{pmatrix}.
    \]
    Furthermore, if $Z = (Z_1,Z_2)$ satisfies $\det(Z), \left<Z_1,Z_2\right> \not= 0$, then $r$ is unique and $(A,t)$ is unique up to a sign.
\end{proposition}
\begin{proof}
    First assume that $\det(Z) = 0$, so that we can write $Z = (au, bu)$ for some unitary vector $u \in \C^2$ and for $a,b \in \C$. For $Z = 0$ the claim is trivial. If either $a$ or $b$ is zero, but not both, then we can choose $A \in \U(2)$ that maps the only nonzero column into a positive multiple of $e_1$ and the result follows. Finally, we assume that $a$ and $b$ are both non-zero. In this case, choose $A \in \U(2)$ such that $A(au) = |a|e_1$ and $t \in \T$ such that
    \[
        t^2 = \frac{a|b|}{b|a|}.
    \]
    Then, one can easily check that
    \[
        (tA,t)Z =
            \begin{pmatrix}
              |a| & |b| \\
              0 & 0
            \end{pmatrix}.
    \]

    Let us now assume that $\det(Z) \not= 0$. From the unit vector
    \[
        \begin{pmatrix}
          a \\
          b
        \end{pmatrix} = \frac{Z_1}{|Z_1|},
    \]
    we define
    \[
        A =
        \begin{pmatrix}
          \overline{a} & \overline{b} \\
          -b & a
        \end{pmatrix} \in \SU(2).
    \]
    Then, it follows easily that we have
    \[
        AZ =
        \begin{pmatrix}
            |Z_1| & \frac{1}{|Z_1|}\left<Z_2,Z_1\right> \\
            0 & \frac{1}{|Z_1|} \det(Z)
        \end{pmatrix}.
    \]
    If $s,t \in \T$ are given, then we have
    \[
        \left(
        \begin{pmatrix}
            t & 0 \\
            0 & s
        \end{pmatrix}A, t\right)Z =
        \begin{pmatrix}
            |Z_1| & \frac{t^2}{|Z_1|}\left<Z_2,Z_1\right> \\
            0 & \frac{st}{|Z_1|} \det(Z)
        \end{pmatrix}.
    \]
    Hence, it is enough to choose $s,t \in \T$ so that $r_2 = t^2\left<Z_2,Z_1\right>$ and $r_3 = st\det(Z)$ are both non-negative to complete the existence part with $r_1 = |Z_1|$.

    For the uniqueness, let us assume that $\det(Z),\left<Z_1,Z_2\right> \not= 0$ and besides the identity in the statement assume that we also have
    \[
        (A',t') Z =
        \begin{pmatrix}
          r_1' & r_2' \\
          0 & r_3'
        \end{pmatrix},
    \]
    with the same restrictions. Then, we obtain the identity
    \begin{equation}\label{eq:U2T-r}
        (A'A^{-1},t'\overline{t})
        \begin{pmatrix}
          r_1 & r_2 \\
          0 & r_3
        \end{pmatrix} =
        \begin{pmatrix}
          r_1' & r_2' \\
          0 & r_3'
        \end{pmatrix}.
    \end{equation}
    This implies that $A'A^{-1}$ is a diagonal matrix of the form
    \[
        \begin{pmatrix}
          a & 0 \\
          0 & b
        \end{pmatrix}
    \]
    with $a,b \in \T$. Then, taking the determinant of \eqref{eq:U2T-r} we obtain $abr_1r_3 = r_1'r_3'$, which implies that $ab = 1$. If we now use the identities from the entries in \eqref{eq:U2T-r}, then one can easily conclude that $r = r'$ and $(A',t') = \pm (A,t)$.
\end{proof}

The following result is an immediate consequence.

\begin{corollary}
    Let $\varphi \in L^\infty(D)$ be given. Then, $\varphi$ is $\U(2)\times\T^2$-invariant if and only if for a.e.~$Z \in D$ we have
    \[
        \varphi(Z) =
            \varphi\left(
                \begin{matrix}
                  r_1 & r_2 \\
                  0 & r_3
                \end{matrix}
                \right)
    \]
    where $r=(r_1,r_2,r_3)$ are the (essentially) unique values obtained from $Z$ in Proposition~\ref{prop:U2T-orbits}.
\end{corollary}

\section{Toeplitz operators with $\U(2)\times\T^2$-invariant symbols}
\noindent
As noted in Section~\ref{sec:preliminaries}, for every $\lambda > 3$ the restriction of $\pi_\lambda$ to $\R\times\SU(2)\times\R\times\SU(2)$ is multiplicity-free. We start this section by providing an explicit description of the corresponding isotypic decomposition.

Let us consider the following set of indices
\[
    \rN^2 = \{\nu=(\nu_1,\nu_2) \in \Z^2 : \nu_1 \geq \nu_2 \geq 0 \}.
\]
Then, for every $\nu \in \rN^2$, we let $F_\nu$ denote the complex irreducible $\SU(2)$-module with dimension $\nu_1-\nu_2 + 1$. For example, $F_\nu$ can be realized as the $\SU(2)$-module given by $\mathrm{Sym}^{\nu_1-\nu_2}(\C^2)$ or by the space of homogeneous polynomials in two complex variables and degree $\nu_1-\nu_2$. Next, we let the center $\T I_2$ of $\U(2)$ act on the space $F_\nu$ by the character $t \mapsto t^{\nu_1 + \nu_2}$. It is easy to check that the actions on $F_\nu$ of $\SU(2)$ and $\T I_2$ are the same on their intersection $\{\pm I_2\}$. This turns $F_\nu$ into a complex irreducible $\U(2)$-module. We note (and will use without further remarks) that the $\U(2)$-module structure of $F_\nu$ can be canonically extended to a module structure over $\GL(2,\C)$.

We observe that the dual $F_\nu^*$ as $\U(2)$-module is realized by the same space with the same $\SU(2)$-action but with the action of the center $\T I_2$ now given by the character $t \mapsto t^{-\nu_1-\nu_2}$.

If $V$ is any $\R\times\SU(2)\times\R\times\SU(2)$-module, then for every $\lambda$ we consider a new $\R\times\SU(2)\times\R\times\SU(2)$-module given by the action
\begin{equation}\label{eq:Vlambda}
    (x,A,y,B)\cdot v = e^{i\lambda(y-x)} (x,A,y,B)v
\end{equation}
where $(x,A,y,B) \in \R\times\SU(2)\times\R\times\SU(2)$, $v \in V$ and the action of $(x,A,y,B)$ on $v$ on the left-hand side is given by the original structure of $V$. We will denote by $V_\lambda$ this new $\R\times\SU(2)\times\R\times\SU(2)$-module structure.

In particular, for every $\nu \in \rN^2$ the space $F_\nu^*\otimes F_\nu$ is an irreducible module over $\U(2)\times\U(2)$ and, for every $\lambda > 3$, the space $(F_\nu^*\otimes F_\nu)_\lambda$ is an irreducible module over $\R\times\SU(2)\times\R\times\SU(2)$. Note that two such modules defined for $\nu, \nu' \in \rN^2$ are isomorphic (over the corresponding group) if and only if $\nu=\nu'$.

\begin{proposition}\label{prop:U2U2-isotypic}
    For every $\lambda > 3$, the isotypic decomposition of the restriction of $\pi_\lambda$ to $\R\times\SU(2)\times\R\times\SU(2)$ is given by
    \[
        \cH^2_\lambda(D) \cong
                \bigoplus_{\nu \in \rN^2} (F_\nu^*\otimes F_\nu)_\lambda,
    \]
    and this decomposition is multiplicity-free. With respect to this isomorphism and for every $d \in \N$, the subspace $\cP^d(M_{2\times2}(\C))$ corresponds to the sum of the terms for $\nu$ such that $|\nu| = d$.
    Furthermore, for the Cartan subalgebra given by the diagonal matrices of $\fu(2)\times\fu(2)$ and a suitable choice of positive roots, the irreducible $\R\times\SU(2)\times\R\times\SU(2)$-submodule of $\cH^2_\lambda(D)$ corresponding to $(F_\nu^*\otimes F_\nu)_\lambda$ has a highest weight vector given by
    \[
        p_\nu(Z) = z_{11}^{\nu_1-\nu_2}\det(Z)^{\nu_2},
    \]
    for every $\nu \in \rN^2$.
\end{proposition}
\begin{proof}
    By the remarks in Section~\ref{sec:preliminaries} we can consider the representation $\pi_\lambda'$. Furthermore, it was already mentioned in that section that $\cP^d(M_{2\times2}(\C))$ is $\R\times\SU(2)\times\R\times\SU(2)$-invariant and so we compute its decomposition into irreducible submodules. In what follows we consider both $\pi_\lambda$ and $\pi_\lambda'$ always restricted to $\R\times\SU(2)\times\R\times\SU(2)$. We also recall that for $\pi_\lambda'$ we already have an action for $\U(2)\times\U(2)$ without the need of passing to the universal covering group.

    Note that the representation $\pi_\lambda'$ on each $\cP^d(M_{2\times2}(\C))$ naturally extends with the same expression from $\U(2)\times\U(2)$ to $\GL(2,\C)\times\GL(2,\C)$. This action is regular in the sense of representations of algebraic groups. By the Zariski density of $\U(2)$ in $\GL(2,\C)$ it follows that invariance and irreducibility of subspaces as well as isotypic decompositions with respect to either $\U(2)$ or $\GL(2,\C)$ are the same for $\pi_\lambda'$ in $\cP^d(M_{2\times2}(\C))$. Hence, we can apply Theorem~5.6.7 from \cite{GW} (see also \cite{Johnson80}) to conclude that
    \[
        \cP^d(M_{2\times2}(\C)) \cong
                \bigoplus_{\substack{\nu \in \rN^2 \\ |\nu| = d}} F_\nu^*\otimes F_\nu
    \]
    as $\U(2)\times\U(2)$-modules for the representation $\pi_\lambda'$. Since the representations $\pi_\lambda$ and $\pi_\lambda'$ differ by the factor $e^{i\lambda(y-x)}$ for elements of the form $(x,A,y,B)$, taking the sum over $d \in \N$ we obtain the isotypic decomposition of $\cH^2_\lambda(D)$ as stated. This is multiplicity-free as a consequence of the remarks in this section.

    Finally, the claim on highest weight vectors is contained in the proof of Theorem~5.6.7 from \cite{GW}, and it can also be found in \cite{Johnson80}.
\end{proof}

We now consider the subgroup $\U(2)\times\T^2$. Note that the subgroup of $\R\times\SU(2)\times\R\times\SU(2)$ corresponding to $\U(2)\times\T^2$ is realized by $\R\times\SU(2)\times\R\times\T$ with covering map given by the expression
\[
    (x,A,y,t) \mapsto
        \left(e^{ix}A,
            e^{iy}
            \begin{pmatrix}
              t & 0 \\
              0 & \overline{t}
            \end{pmatrix}
            \right).
\]
In particular, the action of $\R\times\SU(2)\times\R\times\T$ on $D$ is given by
\[
    (x,A,y,t) Z = e^{i(x-y)}AZ
        \begin{pmatrix}
          t & 0 \\
          0 & \overline{t}
        \end{pmatrix},
\]
and the representation $\pi_\lambda$ restricted to $\R\times\SU(2)\times\R\times\T$ is given by
\[
    (\pi_\lambda(x,A,y,t)f)(Z) = e^{i\lambda(y-x)}
        f\left(e^{i(y-x)}A^{-1}Z
            \begin{pmatrix}
              t & 0 \\
              0 & \overline{t}
            \end{pmatrix}\right).
\]

We recall that for any Cartan subgroup of $\U(2)$ we have a weight space decomposition
\[
    F_\nu = \bigoplus_{j=0}^{\nu_1-\nu_2} F_\nu(\nu_1-\nu_2 - 2j),
\]
where $F_\nu(k)$ denotes the $1$-dimensional weight space corresponding to the weight $k = -\nu_1+\nu_2, -\nu_1+\nu_2 +2, \dots, \nu_1-\nu_2 - 2, \nu_1-\nu_2$. For simplicity, we will always consider the Cartan subgroup $\T^2$ of $\U(2)$ given by its subset of diagonal matrices. We conclude that $F_\nu(k)$ is isomorphic, as a $\T^2$-module, to the $1$-dimensional representation corresponding to the character $(t_1,t_2)\mapsto t_1^{\nu_2}t_2^k$. We will denote by $\C_{(m_1,m_2)}$ the $1$-dimensional $\T^2$-module defined by the character $(t_1,t_2) \mapsto t_1^{m_1}t_2^{m_2}$, where $(m_1,m_2) \in \Z^2$. In particular, we have $F_\nu(k) \cong \C_{(\nu_2,k)}$ for every $k = -\nu_1+\nu_2, -\nu_1+\nu_2 +2, \dots, \nu_1-\nu_2 - 2, \nu_1-\nu_2$.

Using the previous notations and remarks we can now describe the isotypic decomposition for the restriction of $\pi_\lambda$ to $\R\times\SU(2)\times\R\times\T$. As before, for a module $V$ over the group $\R\times\SU(2)\times\R\times\T$ we will denote by $V_\lambda$ the module over the same group obtained by the expression \eqref{eq:Vlambda}.

\begin{proposition}\label{prop:U2T2-isotypic}
    For every $\lambda > 3$, the isotypic decomposition of the restriction of $\pi_\lambda$ to $\R\times\SU(2)\times\R\times\T$ is given by
    \[
        \cH^2_\lambda(D) \cong
            \bigoplus_{\nu \in \rN^2}
                \bigoplus_{j=0}^{\nu_1-\nu_2}
                    (F_\nu^*\otimes \C_{(\nu_2,\nu_1-\nu_2-2j)})_\lambda,
    \]
    and this decomposition is multiplicity-free.
    Furthermore, for the Cartan subalgebra given by the diagonal matrices of $\fu(2)\times i\R^2$ and a suitable choice of positive roots, the irreducible $\R\times\SU(2)\times\R\times\T$-submodule of $\cH^2_\lambda(D)$ corresponding to $(F_\nu^*\otimes \C_{(\nu_2,\nu_1-\nu_2-2j)})_\lambda$ has a highest weight vector given by
    \[
        p_{\nu,j}(Z) = z_{11}^{\nu_1-\nu_2-j}z_{12}^j\det(Z)^{\nu_2},
    \]
    for every $\nu \in \rN^2$ and $j = 0, \dots, \nu_1-\nu_2$.
\end{proposition}
\begin{proof}
    We build from Proposition~\ref{prop:U2U2-isotypic} and its proof so we follow their notation.

    As noted above in this section we have a weight space decomposition
    \[
        F_\nu = \bigoplus_{j=0}^{\nu_1-\nu_2} F_\nu(\nu_1-\nu_2 - 2j)
            \cong \bigoplus_{j=0}^{\nu_1-\nu_2} \C_{(\nu_2,\nu_1-\nu_2 - 2j)},
    \]
    where the isomorphism holds term by term as modules over the Cartan subgroup $\T^2$ of diagonal matrices of $\U(2)$. It follows from this and Proposition~\ref{prop:U2U2-isotypic} that we have an isomorphism
    \[
        \cH^2_\lambda(D) \cong
            \bigoplus_{\nu \in \rN^2}
                \bigoplus_{j=0}^{\nu_1-\nu_2}
                    F_\nu^*\otimes \C_{(\nu_2,\nu_1-\nu_2-2j)},
    \]
    of modules over $\U(2)\times\T^2$ for the restriction of $\pi_\lambda'$ to this subgroup. Hence, with the introduction of the factor $e^{i\lambda(y-x)}$ from \eqref{eq:Vlambda} we obtain the isomorphism of modules over $\R\times\SU(2)\times\R\times\T$ for the restriction of $\pi_\lambda$ to this subgroup. This proves the first part of the statement.

    We also note that the modules $(F_\nu^*\otimes \C_{(\nu_2,\nu_1-\nu_2-2j)})_\lambda$ are clearly irreducible over $\R\times\SU(2)\times\R\times\T$ and non-isomorphic for different values of $\nu$ and $j$. Hence, the restriction of $\pi_\lambda$ to $\R\times\SU(2)\times\R\times\T$ is multiplicity-free.

    On the other hand, the proof of Theorem~5.6.7 from \cite{GW}, on which that of Proposition~\ref{prop:U2U2-isotypic} is based, considers the Cartan subalgebra defined by diagonal matrices in $\fu(2)\times\fu(2)$ and the order on roots for which the positive roots correspond to matrices of the form $(X,Y)$ with $X$ lower triangular and $Y$ upper triangular. With these choices, for every $\nu \in \rN^2$, the highest weight vector $p_\nu(Z)$ from Proposition~\ref{prop:U2U2-isotypic} lies in the subspace corresponding to the tensor product of two highest weight spaces. Hence, $p_\nu(Z)$ lies in the subspace corresponding to $(F_\nu^*\otimes \C_{(\nu_2,\nu_1-\nu_2)})_\lambda$. In particular, $p_\nu(Z)$ is a highest weight vector for $(F_\nu^*\otimes \C_{(\nu_2,\nu_1-\nu_2)})_\lambda$.

    It is well known from the description of the representations of $\fsl(2,\C)$ that the element
    \[
        Y =
        \begin{pmatrix}
          0 & 0 \\
          1 & 0
        \end{pmatrix} \in \fsl(2,\C)
    \]
    acts on $F_\nu$ so that it maps
    \[
        F_\nu(\nu_1-\nu_2-2j) \rightarrow F_\nu(\nu_1-\nu_2-2j-2)
    \]
    isomorphically for every $j = 0, \dots, \nu_1-\nu_2-1$. This holds for the order where the upper triangular matrices in $\fsl(2,\C)$ define positive roots. Since the action of $\U(2)\times\{I_2\}$ commutes with that of $Y$ it follows that the element $(0,Y) \in \fsl(2,\C)\times\fsl(2,\C)$ maps a highest weight vector of $F_\nu^*\otimes \C_{(\nu_2,\nu_1-\nu_2-2j)}$ onto a highest weight vector of $F_\nu^*\otimes \C_{(\nu_2,\nu_1-\nu_2-2j-2)}$. Hence, a straightforward computation that applies $j$-times the element $(0,Y)$ starting from $p_\nu(Z)$ shows that the vector
    \[
        p_{\nu,j}(Z) = z_{11}^{\nu_1-\nu_2-j}z_{12}^j\det(Z)^{\nu_2}
    \]
    defines a highest weight vector for the submodule corresponding to space $F_\nu^*\otimes \C_{(\nu_2,\nu_1-\nu_2-2j)}$ for the representation $\pi_\lambda'$ restricted to $\R\times\SU(2)\times\R\times\T$. Again, it is enough to consider the factor from \eqref{eq:Vlambda} to conclude the claim on the highest weight vectors for $\pi_\lambda$ restricted to $\R\times\SU(2)\times\R\times\T$.
\end{proof}

As a consequence we obtain the following result.

\begin{theorem}\label{thm:Toeplitz-U2T2}
    For every $\lambda > 3$, the $C^*$-algebra $\cT^{(\lambda)}(\cA^{\U(2)\times\T^2})$ generated by Toeplitz operators with essentially bounded $\U(2)\times\T^2$-invariant symbols is commutative. Furthermore, if $H$ is a connected subgroup between $\T^4$ and $\U(2)\times\U(2)$ such that $\cT^{(\lambda)}(\cA^H)$ is commutative, then $H$ is either of $\U(2)\times\U(2)$, $\U(2)\times\T^2$ or $\T^2\times\U(2)$. Also, for the last two choices of $H$, the corresponding $C^*$-algebras $\cT^{(\lambda)}(\cA^H)$ are unitarily equivalent.
\end{theorem}
\begin{proof}
    The commutativity of $\cT^{(\lambda)}(\cA^{\U(2)\times\T^2})$ follows from Proposition~\ref{prop:U2T2-isotypic} and Theorem~\ref{thm:H-commutativeC*}. The possibilities on the choices of $H$ follows from Proposition~\ref{prop:subgroupsT4U2U2} and the remarks from Section~\ref{sec:preliminaries}. The last claim is the content of Proposition~\ref{prop:U2T2vsT2U2}.
\end{proof}

We also obtain the following orthogonality relations for the polynomials $p_{\nu,j}$.

\begin{proposition}\label{prop:pnuj_Schur_relations}
    Let $\nu \in \rN^2$ be fixed. Then, we have
    \[
        \int_{\U(2)} p_{\nu,j}(A) \overline{p_{\nu,k}(A)} \dif A
            = \frac{\delta_{jk}}{\nu_1-\nu_2+1} \binom{\nu_1-\nu_2}{j}
    \]
    for every $j,k = 0, \dots, \nu_1-\nu_2$.
\end{proposition}
\begin{proof}
    We remember that the irreducible $\U(2)$-module $F_\nu$ can be realized as the space of homogeneous polynomials of degree $\nu_1-\nu_2$ in two complex variables. For this realization, the $\U(2)$-action is given by
    \[
        (\pi_\nu(A)p)(z) = \det(A)^{\nu_1}p(A^{-1} z)
    \]
    for $A \in U(2)$ and $z \in \C^2$.

    Also, the computation of orthonormal bases on Bergman spaces on the unit ball (see for example \cite{Zhu2005}) implies that there is a $\U(2)$-invariant inner product $\left<\cdot,\cdot\right>$ on $F_\nu$ for which the basis
    \[
        \left\{ v_j(z_1,z_2) = \binom{\nu_1-\nu_2}{j}^{\frac{1}{2}} z_1^{\nu_1-\nu_2-j}z_2^{j}  :  j=0,1,\ldots, \nu_1-\nu_2\right\},
    \]
 is orthonormal. We fix the inner product and this orthonormal basis for the rest of the proof.

    With these choices it is easy to see that the map given by
    \[
        Z \mapsto \left<\pi_\nu(Z)v_j, v_0\right>,
    \]
    for $Z \in \GL(2,\C)$, is polynomial and is a highest weight vector for the $\U(2)\times\T^2$-module corresponding to $F_\nu^*\otimes \C_{(\nu_2,\nu_1-\nu_2-2j)}$ in the isomorphism given by Proposition~\ref{prop:U2T2-isotypic}. Hence there is a complex number $\alpha_{\nu,j}$ such that
    \[
        p_{\nu,j}(Z) =
            \alpha_{\nu,j} \left<\pi_\nu(Z)v_j, v_\nu\right>
    \]
    for all $Z \in \GL(2,\C)$ and $j = 0, \dots, \nu_1-\nu_2$.

    By Schur's orthogonality relations we conclude that
    \[
        \int_{\U(2)} p_{\nu,j}(Z) \overline{p_{\nu,k}(Z)} \dif Z
            = \frac{\delta_{jk}|\alpha_{\nu,j}|^2}{\nu_1-\nu_2+1}
    \]
    for every $j,k = 0, \dots, \nu_1-\nu_2$.

    Next we choose
    \[
        A_0 =
            \begin{pmatrix}
                \frac{1}{\sqrt{2}}  & -\frac{1}{\sqrt{2}} \\
                \frac{1}{\sqrt{2}}  & \frac{1}{\sqrt{2}}
            \end{pmatrix}
              \in \SU(2).
    \]
    and evaluate at this matrix to compute the constant $\alpha_{\nu,j}$.

    First, we compute
    \begin{align*}
        (\pi_\nu (A_0^{-1}) v_0)(z_1,z_2)
            &= v_0 \left(
                \begin{pmatrix}
                    \frac{1}{\sqrt{2}}  & -\frac{1}{\sqrt{2}} \\
                    \frac{1}{\sqrt{2}}  & \frac{1}{\sqrt{2}}
                \end{pmatrix}
                \begin{pmatrix}
                    z_1\\
                    z_2
                \end{pmatrix}
                \right) \\
            &= v_0\left(\frac{1}{\sqrt{2}}(z_1-z_2),\frac{1}{\sqrt{2}}(z_1+z_2)\right) \\
            &= \frac{1}{\sqrt{2^{\nu_1-\nu_2}}} (z_1-z_2)^{\nu_1-\nu_2} \\
            &= \frac{1}{\sqrt{2^{\nu_1-\nu_2}}}
                \sum_{j=0}^{\nu_1-\nu_2} (-1)^j
                \binom{\nu_1-\nu_2}{j}z_1^{\nu_1-\nu_2-j}z_2^j,
    \end{align*}
    which implies that
    \[
        \left<\pi_\nu(A_0)v_j, v_0\right> =
        \left<v_j, \pi_\nu(A_0^{-1})v_0\right> =
        \frac{(-1)^j}{\sqrt{2^{\nu_1-\nu_2}}} \binom{\nu_1-\nu_2}{j}^{\frac{1}{2}}.
    \]
    Meanwhile,
    \[
        p_{\nu,j}(A_0)
        = \left(\frac{1}{\sqrt{2}}\right)^{\nu_1-\nu_2-j} \left(-\frac{1}{\sqrt{2}}\right)^{j} \det(A_0)^{\nu_2}
        = \frac{(-1)^j}{\sqrt{2^{\nu_1-\nu_2}}},
    \]
    thus implying that
    \[
        \alpha_{\nu,j} = \binom{\nu_1-\nu_2}{j}^{\frac{1}{2}}.
    \]
    This completes our proof.
\end{proof}

\section{The spectra of Toeplitz operators with $\U(2)\times\T^2$-invariant symbols}\label{sec:spectra}
\noindent
We recall that the Haar measure $\mu$ on $\GL(2,\C)$ is given by
\[
   \dif\mu(Z) =| \det(Z)|^{-4} \dif Z=\det (ZZ^*)^{-2} \dif Z.
\]
where $\dif Z$ denotes the Lebesgue measure on the Euclidean space
$M_{2\times2}(\C)$. Furthermore, we have the following expression for the Haar measure:

\begin{lemma}\label{lem:Haar_foliated}
    For every function $f \in C_c(\GL(2,\C))$ we have
    \[
        \int_{\GL(2,\C)} f(Z) \dif\mu(Z)
        = \int_\C \int_{(0,\infty)^2} \int_{\U(2)}
        f\left( A
            \left( \begin{matrix}
                 a_1 & z\\
                 0   & a_2
                \end{matrix}\right)
            \right)
           a_2^{-2} \dif A \dif a \dif z.
    \]
\end{lemma}
\begin{proof}
For the moment let
\[ n_z =
\begin{pmatrix} 1 & z\\ 0 & 1\end{pmatrix}. \]
 We start with the Iwasawa decomposition of $GL(2,\C)$ that allows us to decompose any $Z \in \GL(2,\C)$ as
    \[
        Z = A \diag (a_1,b_1) n_z     \]
where $A \in \U(2)$, $a_1, a_2 > 0$ and $z \in \C$. Then,
\cite[Prop. 8.43]{Knapp2002} and some changes of coordinates we obtain the result as follows.
    \begin{align*}
        \int_{\GL(2,\C)} f(Z) &\dif\mu(Z) \\
            &=  \int_\C \int_0^\infty \int_0^\infty \int_{\U(2)}
                f\left(A
                   \begin{pmatrix}
                       a_1 & 0 \\
                        0   & a_2
                    \end{pmatrix} n_z
                \right)
                    a_1^2 a_2^{-2} \dif A \dif a_1 \dif a_2 \dif z \\
            &= \int_\C \int_0^\infty \int_0^\infty \int_{\U(2)}
                f\left(A
                    \begin{pmatrix}
                        a_1 & a_1 z\\
                        0   & a_2
                    \end{pmatrix}
                \right)
                    a_1^2 a_2^{-2} \dif A \dif a_1 \dif a_2 \dif z  \\
            &= \int_\C \int_0^\infty \int_0^\infty \int_{\U(2)}
                f\left(A
                    \begin{pmatrix}
                        a_1 & z\\
                        0   & a_2
                    \end{pmatrix}
                \right)
                a_2^{-2} \dif A \dif a_1 \dif a_2 \dif z.\qedhere
    \end{align*}
\end{proof}

By the remarks above, the weighted measure $v_\lambda$ on $D$ can be written in terms of the Haar measure on $\GL(2,\C)$ as follows

\begin{align}
   \dif v_{\lambda}(Z) &= c_\lambda | \det(Z)|^4\det(I_2-ZZ^*)^{\lambda-4} \dif\mu(Z)\label{eq:vlambda_Haar}\\
   &=c_\lambda \det(ZZ^*)^2\det(I_2-ZZ^*)^{\lambda-4} \dif\mu(Z) .\nonumber
\end{align}
We use this and Lemma~\ref{lem:Haar_foliated} to write down the measure $v_\lambda$ in terms of measures associated to the foliation on $M_{2\times2}(\C)$ given by the action of $\U(2)\times \T^2$ (see Proposition \ref{prop:U2T-orbits}). The next result applies only to suitably invariant functions, but this is enough for our purposes.

\begin{proposition}\label{prop:vlambda_U2T2}
    Let $\lambda > 3$ be fixed. If $f \in C_c(M_{2\times2}(\C))$ is a function that satisfies $f(t_\theta Z t_\theta^{-1}) = f(Z)$ for every $Z \in M_{2\times2}(\C)$ where
    \[
        t_\theta =
            \begin{pmatrix}
              e^{2\pi i \theta} & 0 \\
              0 & e^{-2\pi i \theta}
            \end{pmatrix}, \quad \theta\in\R ,
    \]
   then we have
    \[
        \int_{M_{2\times2}(\C)} f(Z) \dif v_\lambda(Z) =
          2\pi c_\lambda \int_{R_+^3} \int_{\U(2)}
            f\left(A
                \begin{pmatrix}
                    r_1 & r_2\\
                    0   & r_3
                \end{pmatrix}
            \right) r_1^4 r_2 r_3^2
            b(r)^{\lambda-4}
                \dif A \dif r,
    \]
    where $b(r) = 1-r_1^2-r_2^2-r_3^2+r_1^2 r_3^2$ for $r \in (0,\infty)^3$.
\end{proposition}
\begin{proof}
    First we observe that for every $A \in U(n), a_1, a_2 > 0$ and $z \in \C$ we have
    \begin{align*}
        \det\left(I_2 - A
            \begin{pmatrix}
              a_1 & z \\
              0  & a_2
            \end{pmatrix}
            \begin{pmatrix}
              a_1 &  0 \\
              \overline{z}  & a_2
            \end{pmatrix}
            A^*\right)
           &=
           \det\left(I_2 -
            \begin{pmatrix}
                a_1^2 + |z|^2 & a_2 z \\
                a_2 \overline{z} & a_2^2
            \end{pmatrix}
           \right) \\
           &=
           \det\left(
            \begin{pmatrix}
                 1-a_1^2 - |z|^2 & -a_2 z \\
                 -a_2 \overline{z} & 1-a_2^2
            \end{pmatrix}\right)  \\
           &=
           1-a_1^2-a_2^2-|z|^2+a_1^2 a_2^2 \\
           &=
           b(a_1,|z|,a_2),
    \end{align*}
    where $b$ is defined as in the statement. Using this last identity, \eqref{eq:vlambda_Haar} and Lemma~\ref{lem:Haar_foliated} we compute the following for $f$ as in the statement. We apply some coordinates changes and use the bi-invariance of the Haar measure of $\U(n)$.
    \begin{align*}
        \int_{M_{2\times2}(\C)} &f(Z) \dif v_\lambda(Z) \\
            =&\,
               c_\lambda \int_\C \int_{(0,\infty)^2} \int_{\U(2)}
                f\left(A
                    \begin{pmatrix}
                        a_1 & z\\
                        0   & a_2
                    \end{pmatrix}
                \right) \\
                &\times a_1^4 a_2^{2}
                    b(a_1,|z|,a_2)^{\lambda-4}
                    \dif A \dif a \dif z \\
            =&\,
               2\pi c_\lambda \int_0^1 \int_{(0,\infty)^3} \int_{\U(2)}
                f\left(A
                    \begin{pmatrix}
                        a_1 & re^{2\pi i \theta}\\
                        0   & a_2
                    \end{pmatrix}
                \right) \\
                &\times a_1^4 a_2^{2} r
                    b(a_1,r,a_2)^{\lambda-4}
                    \dif A \dif a \dif r \dif \theta \\
            =&\,
               2\pi c_\lambda \int_0^1 \int_{(0,\infty)^3} \int_{\U(2)}
                f\left(A t_{\theta/2}
                    \begin{pmatrix}
                        a_1 & r\\
                        0   & a_2
                    \end{pmatrix} t_{\theta/2}^{-1}
                \right) \\
                &\times a_1^4 a_2^{2} r
                    b(a_1,r,a_2)^{\lambda-4}
                    \dif A \dif a \dif r \dif \theta \\
            =&\,
               2\pi c_\lambda \int_0^1 \int_{(0,\infty)^3} \int_{\U(2)}
                f\left(t_{\theta/2}^{-1}A t_{\theta/2}
                    \begin{pmatrix}
                        a_1 & r\\
                        0   & a_2
                    \end{pmatrix}
                \right) \\
                &\times a_1^4 a_2^{2} r
                    b(a_1,r,a_2)^{\lambda-4}
                    \dif A \dif a \dif r \dif \theta \\
            =&\,
               2\pi c_\lambda \int_0^1 \int_{(0,\infty)^3} \int_{\U(2)}
                f\left(A
                    \begin{pmatrix}
                        a_1 & r\\
                        0   & a_2
                    \end{pmatrix}
                \right) \\
                &\times a_1^4 a_2^{2} r
                    b(a_1,|z|,a_2)^{\lambda-4}
                    \dif A \dif a \dif r \dif \theta.
    \end{align*}
\end{proof}

In view of Proposition~\ref{prop:vlambda_U2T2} the following formula will be useful.

\begin{lemma}\label{lem:pnuj_asum}
    For every $\nu \in \rN^2$ and $j = 0, \dots, \nu_1-\nu_2$ we have
    \[
        p_{\nu,j}\left(A
            \begin{pmatrix}
              r_1 & r_2 \\
              0 & r_3
            \end{pmatrix} \right)
            = \sum_{k=0}^j \binom{j}{k}
                p_{\nu,k}(A) r_1^{\nu_1-j}r_2^{j-k} r_3^{\nu_2 + k}
    \]
    for every $A \in \U(2)$ and $r \in (0,\infty)^3$.
\end{lemma}
\begin{proof}
    Let $A \in \U(2)$ be given and write
    \[
        A = \begin{pmatrix}
                \alpha & \beta \\
                -\gamma \overline{\beta}  & \gamma \overline{ \alpha}
            \end{pmatrix},
    \]
    where $\alpha,\beta,\gamma\in\C$ with $|\alpha|^2 + |\beta| ^2 = 1$ and $|\gamma| = 1$. Hence, we have
    \[
        A
        \begin{pmatrix}
          r_1  & r_2 \\
          0    & r_3
        \end{pmatrix} =
            \begin{pmatrix}
                \alpha r_1 &  \alpha r_2 + \beta r_3 \\
                *    &  *
            \end{pmatrix},
    \]
    and so we conclude that
    \begin{align*}
        p_{\nu,j}&\left(A
            \begin{pmatrix}
                r_1  & r_2 \\
                0    & r_3
            \end{pmatrix}
            \right) \\
            &= (\alpha r_1)^{\nu_1-\nu_2-j} (\alpha r_2 + \beta r_3)^j
                \det
                \left(A
                    \begin{pmatrix}
                        r_1  & r_2 \\
                        0    & r_3
                    \end{pmatrix}
                \right)^{\nu_2} \\
            &= (\alpha r_1)^{\nu_1-\nu_2-j}
                \sum_{k=0}^j \binom{j}{k}(\alpha r_2)^{j-k} (\beta r_3)^k
                \det
                \left(A
                    \begin{pmatrix}
                        r_1  & r_2 \\
                        0    & r_3
                    \end{pmatrix}
                \right)^{\nu_2} \\
            &= \sum_{k=0}^j \binom{j}{k} \alpha ^{\nu_1-\nu_2-k} \beta^{k}
                \det(A)^{\nu_2} r_1^{\nu_1-j}r_2^{j-k} r_3^{\nu_2 + k} \\
            &= \sum_{k=0}^j \binom{j}{k}
                p_{\nu,k}(A) r_1^{\nu_1-j}r_2^{j-k} r_3^{\nu_2 + k}.
    \end{align*}
    Note that in the last line we have used the expression obtained in the first line.
\end{proof}

We now apply the previous results to compute the spectra of the Toeplitz operators with $\U(2)\times\T^2$-invariant symbols.

\begin{theorem}\label{thm:coefficients}
    Let $\lambda > 3$ and $\varphi \in \cA^{\U(2)\times\T^2}$ be given. With the notation of Proposition~\ref{prop:U2T2-isotypic}, the Toeplitz operator $T_\varphi$ acts on the subspace of $\cH^2_\lambda(D)$ corresponding to $(F_\nu^*\otimes \C_{(\nu_2,\nu_1-\nu_2-2j)})_\lambda$ as a multiple of the identity by the constant
    \begin{align*}
        \gamma(\varphi,\nu,j) &=
            \frac{\left<\varphi p_{\nu,j},p_{\nu,j}\right>_\lambda}%
            {\left<p_{\nu,j},p_{\nu,j}\right>_\lambda} = \\
            &\frac{\displaystyle\sum_{k=0}^j \binom{j}{k}^2\binom{\nu_1-\nu_2}{k}
                    \int_\Omega
                        \varphi
                            \begin{pmatrix}
                                r_1 & r_2 \\
                                0 & r_3 \\
                            \end{pmatrix}
                        a(r,\nu, j, k) b(r)^{\lambda-4} \dif r}%
            {\displaystyle\sum_{k=0}^j \binom{j}{k}^2\binom{\nu_1-\nu_2}{k}
                    \int_\Omega
                        a(r,\nu, j, k) b(r)^{\lambda-4} \dif r}
    \end{align*}
    for every $\nu\in\rN^2$ and $j=0,\dots,\nu_1-\nu_2$, where
    \[
        \Omega = \left\{ r \in (0,\infty)^3 :
                  \begin{pmatrix}
                    r_1 & r_2 \\
                    0 & r_3
                  \end{pmatrix} \in D \right\}.
    \]
    with the functions $a(r,\nu, j, k)= r_1^{2(\nu_1-j)+4}r_2^{2(j-k)+1} r_3^{2(\nu_2+k)+2}$, for $0 \leq k \leq j$, and $b(r) = 1-r_1^2-r_2^2-r_3^2+r_1^2 r_3^2$ for $r \in (0,\infty)^3$.
\end{theorem}
\begin{proof}
    Let $\varphi \in \cA^{\U(2)\times\T^2}$ be given and fix $\nu \in \rN^2$ and $j = 0, \dots, \nu_1-\nu_2$.  First, we observe that we have
    \[
        |p_{\nu,j}(tZt^{-1})|^2 = |p_{\nu,j}(Z)|^2
    \]
    for all $Z\in M_{2\times2}(\C)$ and $t\in\T^2$. The symbol $\varphi$ is bi-$\T^2$-invariant as well. Hence, we can apply Proposition~\ref{prop:vlambda_U2T2} to $\varphi |p_{\nu,j}|^2$ to compute as follows
    \begin{align*}
        \left<\varphi p_{\nu,j},p_{\nu,j}\right>_\lambda
        &= \int_D \varphi(Z) |p_{\nu,j}(Z)|^2 \dif v_\lambda(Z) \\
        =&\,  2\pi c_\lambda
        \int_\Omega \int_{U(2)}
            \varphi\left(A
                \begin{pmatrix}
                  r_1 & r_2 \\
                  0 & r_3
                \end{pmatrix}
            \right)
            \left|p_{\nu,j}\left(A
                \begin{pmatrix}
                  r_1 & r_2 \\
                  0 & r_3
                \end{pmatrix}
            \right)\right|^2  \\
        &\times r_1^4 r_2 r_3^2 b(r)^{\lambda-4}
        \dif A \dif r \\
        =&\, 2\pi c_\lambda
            \sum_{k=0}^j \binom{j}{k}^2 \int_\Omega
            \varphi
                \begin{pmatrix}
                  r_1 & r_2 \\
                  0 & r_3
                \end{pmatrix}
            \int_{U(2)} |p_{\nu,k}(A)|^2 \dif A \\
        &\times a(r,\nu,j.k) b(r)^{\lambda-4} \dif r \\
        =&\, \frac{2\pi c_\lambda}{\nu_1-\nu_2+1}
            \sum_{k=0}^j \binom{j}{k}^2\binom{\nu_1-\nu_2}{k}
            \int_\Omega
                \varphi
                \begin{pmatrix}
                  r_1 & r_2 \\
                  0 & r_3
                \end{pmatrix}\\
                & \times
            a(r,\nu,j.k) b(r)^{\lambda-4} \dif r.
    \end{align*}
    The second identity applies Proposition~\ref{prop:vlambda_U2T2}. For the third identity we apply Proposition~\ref{prop:pnuj_Schur_relations} and the invariance of $\varphi$. In the last identity we apply again the orthogonality relations from Proposition~\ref{prop:pnuj_Schur_relations}.

    The proof is completed by taking $\varphi \equiv 1$ in the above computation.
\end{proof}

\end{document}